\documentclass{amsart}
\usepackage{graphicx} 
\usepackage{amsfonts,amssymb,amsmath,amsthm,autobreak, cancel,dsfont,mathtools,mathbbol,mathrsfs,siunitx,upgreek}

\usepackage{booktabs,longtable,dcolumn,makecell,
    multicol,multirow,tabularx,xltabular,rotating}

\newtheorem{thm}{Theorem}[section]
\newtheorem{cor}[thm]{Corollary}
\newtheorem{prop}[thm]{Proposition}
\newtheorem{lem}[thm]{Lemma}
\newtheorem{conj}[thm]{Conjecture}

\theoremstyle{definition}
\newtheorem{defn}[thm]{Definition}
\newtheorem{asp}[thm]{Assumption}

\theoremstyle{remark}
\newtheorem{rem}[thm]{Remark}

\newcommand{\RN}[1]{%
	\textup{\uppercase\expandafter{\romannumeral#1}}%
}
\newcounter{x}

\newcommand{\norm}[2]{L^{#1}_tL^{#2}_x}
\newcommand{\hnorm}[1]{L^{#1}_{t,x}}
\newcommand{\hs}{\mathbb{H}^d}

\numberwithin{equation}{section}

\title{Almost sharp global wellposedness and scattering for the defocusing conformal wave equation on the hyperbolic space }
\author{Chutian Ma}
\address{Department of Mathematics, Johns Hopkins University, MD 21218}
\email{cma27@jhu.edu}
\date{December 8, 2024}

\begin{document}
\begin{abstract}
    In this paper we prove a global well-posedness and scattering result for the defocusing conformal nonlinear wave equation in the hyperbolic space $\mathbb{H}^d, d \geq 3$. We take advantage of the hyperbolic geometry which yields stronger Morawetz and Strichartz estimates. We show that the solution is globally wellposed and scatters if the initial data is radially symmetric and lies in $H^{\frac{1}{2}+\epsilon}(\mathbb{H}^d)\times H^{-\frac{1}{2}+\epsilon}(\mathbb{H}^d)$, $\epsilon>0$. 
\end{abstract}
\maketitle

\section{Introduction}

In this paper, we study the low regularity global well-posedness problem for the defocusing nonlinear wave equation on the hyperbolic space defined below:
\begin{equation}\label{eq:nlw}
    \left\{\begin{aligned}
        & u_{tt} - \Delta_{\mathbb{H}^d}u + \mu |u|^{p-1}u = 0 \\
        & u(0,x) = u_0 \\
        & u_t(0,x) = u_1 \\
    \end{aligned}\right.
\end{equation}
The term $\mu |u|^{p-1} u$ is known as non-linearity of power type. The equations are called defocusing when $\mu = 1$, and focusing when $\mu = -1$. \\

In this paper, we consider strong $H^s$ solutions as defined in \cite{tao2006nonlinear}:
\begin{defn}[Strong $H^s$ solutions]
    Given initial data $(u_0, u_1) \in H^s \times H^{s-1}$ where $H^s, H^{s-1}$ are the usual Sobolev spaces on the hyperbolic space, we say that u is a strong $H^s$ solution to the wave equation (\ref{eq:nlw}) on time interval I if $u \in C^0_{t, loc} H^s_x(I\times \hs)$, $u_t \in C^0_{t, loc} H^{s-1}_x(I\times \hs)$ solves (\ref{eq:nlw}) as distribution, i.e.
    \begin{equation}
        \begin{aligned}
            u(t) = cos(t\sqrt{-\Delta_{\hs}}) u_0 + \frac{sin(t\sqrt{-\Delta_{\hs}})}{\sqrt{-\Delta_{\hs}}} u_1 - \mu \int_0^t \frac{sin((t-s)\sqrt{-\Delta_{\hs}}}{\sqrt{-\Delta_{\hs}}} |u|^{p-1}(s) u(s) ds
        \end{aligned}
    \end{equation}
\end{defn}

In addition to the relative weak concept of strong solution, well-posedness is usually defined and expected given "good" initial data. In short, a well-posed solution depends continuously on the initial values. To make it precise:
\begin{defn}[Well-posedness]
    We say the initial value problem is locally well-posed in $H^s\times H^{s-1}$ if for any $(u_0,u_1)\in H^s\times H^{s-1}$ there exist $T>0$ and a neighborhood U of $(u_0,u_1)$ in $H^s\times H^{s-1}$, such that for any $(u_0',u_1')\in U$, there exists a unique solution u in $C_t^0H_x^s([-T,T]\times \hs)\times C^0_tH_x^{s-1}([-T,T]\times \hs)$. Furthermore, the map $u_0',u_1'\rightarrow u$ is continuous in U. If $T>0$ can be chosen arbitrarily large (we do not require T to be $+\infty$), then we say the initial value problem is globally well-posed. \\
\end{defn}
A natural question is how the solution evolves, including whether the solution is locally and globally well-posed, and its asymptotic behavior. For initial data with sufficient regularity and decay, solutions of nonlinear equations are known to converge to a linear solution as $t\rightarrow \infty$. This phenomenon is called scattering:
\begin{defn}
    (Scattering). Suppose u solves (\ref{eq:nlw}) globally in time. We say that u scatters in $H^s\times H^{s-1}$ if there exists $(u_0^+,u_1^+)$ and $(u_0^-,u_1^-)$ in $H^s\times H^{s-1}$ such that 
    \begin{equation*}
        \begin{aligned}
            &(u,u_t)-S(t)(u_0^+,u_1^+)\rightarrow 0\ in\ H^s\times H^{s-1}\ as\ t\rightarrow +\infty \\   
            &(u,u_t)-S(t)(u_0^-,u_1^-)\rightarrow 0\ in\ H^s\times H^{s-1}\ as\ t\rightarrow -\infty
        \end{aligned}
    \end{equation*}
    where S(t) is the propagator of the free wave equation.
\end{defn}

\subsection{Main results}
In this paper, we study the defocusing equations (\ref{eq:nlw}) on hyperbolic spaces with radially symmetric initial data. Hyperbolic geometry is known to yield stronger dispersion while the radially symmetric solutions tend to spread out and are less likely to concentrate. As a result, we prove the following global well-posedness and scattering result for the conformal wave equation (\ref{eq:nlw}). Assume the initial data $(u_0,u_1)$ satisfies 
\begin{equation}\label{IVP}
    \begin{aligned}
        &(u_0,u_1)\ is\ radially\ symmetric,\\
        &(u_0,u_1)\in H^{\frac{1}{2}+\epsilon}\times H^{-\frac{1}{2}+\epsilon},\hspace*{2em} \epsilon>0
    \end{aligned}
\end{equation}
then 
\begin{thm}[Global wellposedness and scattering for (\ref{eq:nlw})]\label{Main Thm}
    When $d \geq 3$, the solution to the defocusing conformal wave equation $(\ref{eq:nlw})$, i.e. $\mu = 1$ and $p = 1 + \frac{4}{d-1}$ is globally well-posed and scatters under the assumption $(\ref{IVP})$. Moreover, there exists function $f_\epsilon: \mathbb{R}\rightarrow \mathbb{R}$, such that 
    \begin{equation}
        \|u\|_{L^{\frac{2(d+1)}{d-1}}_{t,x}}\leq f_\epsilon(\|(u_0,u_1)\|_{H^{\frac{1}{2}+\epsilon}\times H^{-\frac{1}{2}+\epsilon}})
    \end{equation}
    In other word, the scattering norm can be controlled by the (non-critical) norm, not the profile.
\end{thm}
\begin{rem}
    This result is almost sharp, in the sense that the regularity restriction in Theorem \ref{Main Thm} can be brought close to the restriction predicted by the scaling symmetry (see (\ref{eq s_c})) by letting $\epsilon \rightarrow 0$. 
\end{rem}
\begin{rem}
    The proof for scattering is simplified on $\hs$ compared to the Euclidean proof, e.g. in \cite{dodson2018global}. This is due to a stronger Morawetz inequality provided by hyperbolic geometry.
\end{rem}


\subsection{Review of Backgrounds}

For wave equation (\ref{eq:nlw}), having $p \geq 2$ is not enough to ensure global scatter solution, or even local solutions. 
Let us review the Euclidean version of (\ref{eq:nlw}):
\begin{equation}\label{eq nlw euclidean}
    u_{tt} - \Delta_{\mathbb{R}^d} u + \mu |u|^{p-1}u = 0
\end{equation}
The equation (\ref{eq nlw euclidean}) is invariant under the following scaling symmetry. Suppose u solves (\ref{eq:nlw}) on the interval $[0, T]$. Denote
\begin{equation}\label{eq scaling}
    u_\lambda (t,x) = \lambda^\frac{2}{p-1} u(\lambda t, \lambda x)
\end{equation}
Then $u_\lambda$ solves (\ref{eq:nlw}) on $[0,\frac{T}{\lambda}]$. Denote
\begin{equation}
    s_c = \frac{d}{2} - \frac{2}{p-1}
\end{equation}
. We have
\begin{equation}
    \begin{aligned}
        & \| u_\lambda (0, \cdot) \|_{H^{s_c}} = \|u\|_{H^{s_c}} \\
        & \| \partial_t u_\lambda (0, \cdot) \|_{H^{s_c-1}} = \|\partial_t u(0,\cdot)\|_{H^{s_c-1}}
    \end{aligned}
\end{equation}
The problem is called $s_c$ critical as the $H^{s_c} \times H^{s_c-1}$ norm is invariant under the scaling transformation. The case where $s > s_c$ is called sub-critical, and $s < s_c$ is called super-critical.

Scaling symmetry implies the following restriction on p and s must be satisfied in order for well-posedness to hold:
\begin{equation}\label{eq s_c}
    s \geq s_c
\end{equation}

The heuristic is that in sub-critical cases, by choosing a small $\lambda$, one may trade longer existence in time at the cost of smaller $H^s \times H^{s-1}$ norm of the initial data. Or vice versa, a larger $H^s \times H^{s-1}$ tends to allow shorter time of existence. In super-critical case, scaling symmetry suggests that evolving small data in short time is as hard as evolving large data in long time. In particular, when $s < s_c$, suppose $(u, u_t) \in H^s \times H^{s-1}$ solves (\ref{eq:nlw}) with maximum interval of existence $[0, T]$. Due to finite speed of propagation, we may assume the data to have compact support. Then scaling symmetry allows us to generate solutions with arbitrary small $H^s$ norm, support and interval of existence. Properly translating and summing yields initial data in $H^s$ which fails to be extended in time at all. The authors of \cite{christ2003ill} constructed ill-posed data by using a small dispersion analysis paired with scaling. They showed for defocusing wave and Schr\"{o}dinger equations, when $s < s_c$, there exists solution with arbitrarily small initial data in $H^s$ norm which grows rapidly in arbitrarily small interval of time, which are counterexamples to well-posedness below critical regularity.

Above the critical regularity, it is conjectured that (\ref{eq s_c}) guarantees global existence in the super-conformal range: $s_c \geq \frac{1}{2}$.
\begin{conj}
    Assume $d \geq 3$ and $s \geq s_c$. Then the initial value problem is conjectured to be globally well-posed in $H^{s} \times H^{s-1}$.
\end{conj}

\begin{rem}
    For focusing wave equations, there is an additional obstacle to well-posedness which is concentration along one light ray. See \cite{lindblad1995existence} for construction of counterexamples. This imposes another restriction on p, s that
    \begin{equation}\label{eq subconformal}
        s \geq s_{conf} = \frac{d+1}{4} - \frac{1}{p-1}
    \end{equation}. This restriction only becomes stricter when $p < 1 + \frac{4}{d-1}$, or in other words $s_c < \frac{1}{2}$. The case where $s_c = s_{conf} = \frac{1}{2}$ is called \textit{conformal} as the equation is invariant under conformal transformations. 
\end{rem}


When $s \geq s_c$, one starts with local well-posedness. In general, local well-posedness results are relatively easy to obtain. They usually follows from Strichartz estimates and standard contraction mapping argument. Extending the solution globally is harder. A useful tool to is the conservation law. Wave equation (\ref{eq:nlw}) has conserved energy, defined by
\begin{equation}\label{eq defn energy}
    E(t)=\int \frac{1}{2}|\nabla u|^2 + \frac{1}{2}u_t^2 + \mu \frac{1}{p+1}u^{p+1}
\end{equation}
, assuming u has the required regularity and decay for (\ref{eq defn energy}) to make sense. In energy subcritical case, local well-posedness results ensure the solution can be extended over a small time interval whose length depends on the $H^1 \times L^2$ norm of the solution. For defocusing equations, i.e. $\mu = 1$, the $H^1 \times L^2$ norm of the solution is controlled by the energy. This immediately proves global well-posedness.

Below energy state, i.e. when the initial data is in $H^s \times H^{s-1}$ for some $s < 1$, the major difficulty to proving Theorem \ref{Main Thm} is the lack of a conserved quantity which controls the $H^s \times H^{s-1}$ norm. Thus local well-posedness alone is not sufficient to extend the solution globally.

\subsection{Review of past results}
In Euclidean spaces, such rough data results are well established for (\ref{eq:nlw}) as well as a wide range of dispersive equations. Bourgain in \cite{bourgain1998scattering} studied cubic defocusing NLS below energy norm using the Fourier truncation method, proving global well-posedness for $s>\frac{11}{13}$. In short, Fourier truncation method splits the initial data into a high frequency part and a low frequency part and evolves them under respective equations. The high frequency evolution is rough but has small size, and will be treated as a perturbation. The low frequency evolution is large but has finite initial energy, which will be treated as a smoothed-up approximating solution. Since the high frequency evolution is small, its interaction with the low frequency evolution in the nonlinear terms is expected to be small, which causes the low frequency energy to be "almost conserved". \\

Various later works have employed similar methods in studying rough data IVP for wave equations. \cite{kenig2000global} proved global well-posedness for cubic defocusing NLW in $\mathbb{R}^3$ for $\frac{3}{4}<s<1$. See also \cite{gallagher2003global}, \cite{bahouri2006global}, \cite{roy2007global}, \cite{dodson2018global}, \cite{dodson2018globalBesov} \cite{dodson2019global}, \cite{dodson2022global}. \\

Despite the fact that there is no scaling in hyperbolic spaces, the conjecture is still believed to be true on hyperbolic spaces. In fact, it has long been known that hyperbolic geometry strengthens the dispersion. Heuristically, this is because the volume on hyperbolic space grows exponentially. This fact provides larger room for the solution to disperse into.

For the Schr\"{o}dinger equation, Banica in \cite{banica2007nonlinear} proved dispersive estimates on $\hs$ which are stronger away from the origin compared to the Euclidean counterparts. In particular, radial symmetry data tends to spread out and is less likely to concentrate. Banica and Duyckaerts in \cite{banica2007weighted} proved Strichartz estimates for Schr\"{o}dinger equation on $\hs$ which indicates stronger decay at spacial infinity. 

The same heuristic holds for wave equations on $\hs$. \cite{pierfelice2008weighted} derived weighted (stronger than Euclidean) Strichartz estimates for radially symmetric wave equations on Damek-Ricci spaces (including $\hs$). Anker in \cite{anker2012wave} studied shifted wave equations on $\hs$ where the Laplacian is shifted to compensate for its spectral gap. Strichartz estimates with a large family of admissible exponents are derived, as well as global well-posedness result for small data as an application. On $\mathbb{H}^3$ in particular, Taylor and Metcalfe in \cite{metcalfe2011nonlinear} studied wave equations and established small data global well-posedness results for a broader range of power compared to that on Euclidean spaces. Tatatu in \cite{tataru2001strichartz} derived dispersive estimates for shifted wave equations on $\hs$ and transferred them back to the Euclidean space in order to prove global well-posedness results. See also \cite{metcalfe2012dispersive}, \cite{anker2014wave}, \cite{hassani2011wave}. 

\subsection{Outline of the paper}

The outline of this paper is stated below.\\
In section 2, we list some preliminary results regarding analysis on $\hs$, as well as proving Strichartz estimate for wave equation with $H^\frac{1}{2} \times H^{-\frac{1}{2}}$ initial data, which is slightly better than that in \cite{anker2012wave}.  \\
In section 3, we construct a series of Morawetz potentials. In fact, some of the Morawetz potential functions we choose are not permitted on Euclidean spaces.  \\
In section 4, we prove Theorem \ref{Main Thm} using Fourier truncation method. We will split the initial data into a small, rough part and a large, smooth part. The small rough data evolves globally in time due to small data result. Thus our goal is to prove the global well-posedness of the evolution of the large smooth data, which satisfies the conformal wave equation with an error term caused by the interaction between the rough and smooth evolution. Our main idea is to utilize the Morawetz potentials to control these interaction terms. 

\section{Preliminaries}

We clarify some of the notations and abbreviation we will use throughout the paper. If $X, Y \geq 0$, we write $X\lesssim Y$ if there exists a global constant C, such that $X\leq CY$. We write $X\sim Y$ if $X\lesssim Y$ and $Y\lesssim X$. When there is no risk of confusion, we will write $\|\cdot\|_p$ instead of $\|\cdot\|_{L^p_x}$.

\subsection{Analysis on $\mathbb{H}^d$}
	Let $\mathbb{R}^{d+1}$ be the standard Minkowski space endowed with the metric $-dx_0^2+dx_1^2+dx_2^2+...+dx_d^2$ and the bilinear form $[x,y]=x_0y_0-x_1y_1-x_2y_2-...-x_dy_d$. Hyperbolic space $\mathbb{H}^d$ is defined as the submanifold satisfying $[x,x]=1$ whose metric is induced from the Minkowski space. \\
	We will use the polar coordinates for our analysis on $\mathbb{H}^d$. We use the pair $(r,\omega)\in \mathbb{R}\times S^{d-1}$ to represent the point $(coshr,sinhr\omega)$ in the previous model. The metric induced is \\
	\begin{equation*}
	    g_{\mathbb{H}^d}=dr^2+sinh^2rd\omega^2	
	\end{equation*}
    where $d\omega^2$ is the standard metric on the sphere $\mathbb{S}^{d-1}$.
	And integrations are represented by
	\begin{equation*}
	    \int_{\mathbb{H}^d}	fdvol=\int_0^\infty\int_{\mathbb{S}^{d-1}}f(r,\omega)sinh^{d-1}rdrd\omega
	\end{equation*}

    The definition can be extended to tensor fields $F$, where $|F|$ is defined as
    \begin{equation}\label{defn tensor metric}
        |F| = g_{\mu_1\lambda_1}...g_{\mu_p\lambda_p}g^{\nu_1\kappa_1}...g^{\nu_q\kappa_q} F^{\mu_1...\mu_p}_{\nu_1...\nu_q} F^{\lambda_1...\lambda_p}_{\kappa_1...\kappa_q}
    \end{equation}

    \subsection{Heat flow based Littlewood Paley operators}
    In order to study the frequency interaction, we need a tool which functions in $\hs$ like Littlewood Paley operators in Euclidean spaces. For that purpose, heat flow based operators have been developed and used widely to localize frequencies on manifolds, see \cite{klainerman2006geometric}, \cite{lawrie2018local}, \cite{staffilani2020high} for examples of their applications. In this paper, we will use the notations from \cite{staffilani2020high}. Below we list the definition and some properties of heat flow based operators.

    The Laplacian-Beltrami operator on $\mathbb{H}^d$ can be written in polar coordinates as
    \begin{equation}
        \Delta_{\mathbb{H}^d}=\partial_r^2+\frac{(d-1)coshr}{sinhr}\partial_r+\frac{1}{sinh^2r}\Delta_{\mathbb{S}^{d-1}}
    \end{equation}
    We recall that the spectrum of $-\Delta_{\mathbb{H}^d}$ is $[\rho^2,+\infty]$ with $\rho=\frac{d-1}{2}$. Thus we have the following Poincar\'e inequality:
    \begin{lem}[Poincar\'e inequality]
        For $0\leq\alpha\leq\beta<\infty$ and smooth function f, we have 
        \begin{equation*}
            \|(-\Delta)^\frac{\alpha}{2} f\|_2\leq \rho^{\alpha - \beta} \|(-\Delta)^\frac{\beta}{2} f\|_2
        \end{equation*}
    \end{lem}

    \begin{defn}[Heat Flow Based Frequency Projection Operators]
        For any $s>0$, we define
        \begin{equation*}
            P_{\geq s}f=e^{s\Delta}f,\hspace*{2em} P_sf=(-s\Delta)e^{s\Delta}f,\hspace*{2em} P_{<s}f=f-P_{\geq s}f        
        \end{equation*}
    \end{defn}

    \begin{lem}[Lemma 2.9 in \cite{staffilani2020high} ]
        Heat flow operators are bounded in $L^p$. For $1<p<+\infty$
        \begin{equation*}
            \|P_{\geq s}f\|_p+\|P_sf\|_p+\|P_{<s}f\|_p\lesssim \|f\|_p
        \end{equation*}
    \end{lem}
    
    \begin{lem}[Bernstein Inequalities, Lemma 2.19 in \cite{staffilani2020high}]\label{lem:hfo1}
        For $0\leq\beta<\alpha<\beta+1$,
        \begin{equation*}
            \begin{aligned}
                &\|(-\Delta)^\beta P_{\leq s}f\|_2\lesssim s^{\alpha-\beta}\|(-\Delta)^\alpha f\|_2\\
                &\|(-\Delta)^\alpha P_{\geq s}f\|_2\lesssim s^{\beta-\alpha}\|(-\Delta)^\beta f\|_2
            \end{aligned}
        \end{equation*}
    \end{lem}

    \subsection{Sobolev spaces on $\mathbb{H}^d$}
    Sobolev spaces on hyperbolic spaces can be defined using (\ref{defn tensor metric}). In fact, for integer $k \geq 0$ and $p > 1$, we can define the norm on $W^{k,p}(\hs)$ to be the complement of $C^\infty(\hs)$ under the norm
    \begin{equation}
        \|f\|_{W^{k,p}(\hs)} = \sum_{j=0}^k \left( \int_{\hs} |\nabla^{(j)} f|^p d vol \right)^{\frac{1}{p}}
    \end{equation}
    where $|\nabla^{j}f|$ is defined as in (\ref{defn tensor metric}).
    Fraction Sobolev spaces $W^{s,p}$ for $s\in \mathbb{R}^+$ are then defined by interpolation.
    Alternatively, one may define Sobolev spaces through the spectral theory of Laplacian.
    For $1<p<+\infty$, $s\in\mathbb{R}$, define the norm $\|\cdot\|_{\tilde{W}^{s,p}}$ to be
    \begin{equation*}
        \|f\|_{\tilde{W}^{s,p}}=\|(-\Delta)^\frac{s}{2}f\|_p
    \end{equation*}
    and Sobolev space $\tilde{W}^{s,p}$ to be the complement of $C^\infty(\hs)$ under this norm.
    We denote $H^s$ to be $W^{s,2}(\hs)$ in this paper. When $s \geq 0 $ and $p > 1$, these two norms are equivalent.
    \begin{equation*}
        \|f\|_{W^{s,p}} \sim \|f\|_{\tilde{W}^{s,p}}
    \end{equation*}
    See \cite{tataru2001strichartz} for details.

    We have Sobolev embedding
    \begin{lem}[Sobolev Embedding on $\hs$]
        If $1\leq p<q<\infty$ and $\frac{1}{q}=\frac{1}{p}-\frac{s}{d}$, then we have the following embedding
        \begin{equation*}
            W^{s,p}\hookrightarrow L^q
        \end{equation*}
    \end{lem}
    \begin{lem}[Morrey's inequality]
        If $p\geq 1,\ s>0$ and $\frac{1}{p}-\frac{s}{d}<0$, then we have the following embedding
        \begin{equation*}
            W^{s,p}\hookrightarrow L^\infty
        \end{equation*}
    \end{lem}
    
    We will need the radial Sobolev embedding, which provides stronger control than the standard Sobolev inequalities away from origin.
    \begin{lem}[Radial Sobolev Embedding]
        Suppose f is radial on $\hs$ with $d \geq 2$, and $\frac{1}{2}<\alpha<2$, we have
        \begin{equation}\label{eq radial sobolev}
            \|sinh^{\frac{d-1}{2}}rf\|_\infty\lesssim_\alpha \|f\|_{H^\alpha}
        \end{equation}
    \end{lem}
   See \cite{staffilani2020high} for proof on $\mathbb{H}^2$. Proof for higher dimension requires little modification.

\subsection{Strichartz Estimates}

In this subsection, we will prove the Strichartz estimate for radial solution to the equation on $\hs$:
\begin{equation}\label{eq for Strichartz}
    \begin{aligned}
        &u_{tt}-\Delta_{\hs} u = F\\
        &u(0,r)=u_0(r),\ \ u_t(0)=u_1(r)
    \end{aligned}
\end{equation}

Denote $\rho=\frac{d-1}{2}$. We can convert (\ref{eq for Strichartz}) to a (perturbed) Klein Gordan equation on $\mathbb{R}^d$ using a change of coordinate trick:

\begin{lem}\label{Lemma : Hd to Rd}
    Suppose $u_0, u_1$ and $F$ are radial, then(\ref{eq for Strichartz}) is equivalent to the following initial value problem on $\mathbb{R}^d$:
    \begin{equation}\label{eq for Rd}
        \begin{aligned}
            &v_{tt}-\Delta_{\mathbb{R}^d}v + \rho^2 v + V(r)v = \phi^{-1}(r)F\\
            &v(0,r)=\phi^{-1}(r)u_0(r),\ \ v_t(0)=\phi^{-1}(r)v_1(r)
        \end{aligned}
    \end{equation}
    where $V(r)=-\frac{(d-1)(d-3)}{4}(\frac{1}{r^2}-\frac{1}{sinh^2r})$
\end{lem}

\begin{proof}
    Set 
    \begin{equation*}
        u(t,r)=\phi(r)v(t,r)
    \end{equation*}
    where $\phi(r)=(rsinh^{-1}r)^\rho$ and plug into (\ref{eq for Strichartz}), we get (\ref{eq for Rd}).

\end{proof}

In light of lemma above, showing Strichartz estimates for wave evolution on $\hs$ is equivalent to showing the Strichartz estimates for the perturbed Klein Gordan flow on $\mathbb{R}^d$. In fact, for standard Klein Gordan flow on $\mathbb{R}^d$, we have the following definition of admissibility and Strichartz estimates.
\begin{defn}[Admissible Pairs for Standard Klein Gordan Flow]
    For $d\geq 3$, $p,q\geq 2$, $0\leq\theta\leq 1$ and $\mu\in \mathbb{R}$, we say a pair $(p,q,d,\theta,\mu)$ is admissible iff
    \begin{equation}
        \begin{aligned}
            &\frac{2}{p}+\frac{d-1+\theta}{q}\leq \frac{d-1+\theta}{2}, \ \ (p,q,d,\theta)\neq (2,\infty,3,0)\\
            &\frac{1}{p}+\frac{d+\theta}{q}=\frac{d+\theta}{2}-\mu\\
        \end{aligned}
    \end{equation}
\end{defn}

\begin{lem}[Strichartz Estimates for Standard Klein Gordan Flow on $\mathbb{R}^d$, see \cite{changxing2004global}]\label{Lemma Strichartz Standard}
    Suppose u(t,x) solves the Klein Gordan equation on $I\times \mathbb{R}^d$
    \begin{equation*}
        \begin{aligned}
            &u_{tt}-\Delta_{\mathbb{R}^d}u+\rho^2 u=F\\
            &u(0)=u_0,\ \ u_t(0)=u_1
        \end{aligned}
    \end{equation*}
    $(p,q,d,\theta,\mu)\ and \ (\Tilde{p},\Tilde{q},d,\theta,1-\mu)$ are admissible, then we have
    \begin{equation}
        \|u\|_{L^p_tL^q_x}+\|(u,u_t)\|_{C^0_tH^{\mu}\times H^{\mu-1}}\lesssim \|(u_0,u_1)\|_{H^\mu\times H^{\mu-1}} + \|F\|_{L^{\tilde{p}'}_tL^{\tilde{q}'}_x}
    \end{equation}    
\end{lem}

In order to prove the Strichartz estimate, we will use the Rodnianski-Schlag technique, which allows us to bypass the dispersive estimates for perturbed equation and derive Strichartz estimates from Kato smoothing estimates of the following type.

\begin{lem}[Kato Smoothing Estimate from \cite{d2015kato}]
    Let V(r) be as in Lemma \ref{Lemma : Hd to Rd}, then we have
    \begin{equation}\label{Kato smoothing}
        \||x|^{-1}e^{it\sqrt{-\Delta+1+V}}f\|_{L^2_{t,x}}\lesssim \|f\|_{H^\frac{1}{2}}
    \end{equation}
\end{lem}
\begin{proof}
    By \cite{d2015kato}, (\ref{Kato smoothing}) is true if the following condition on V is satisfied:\\
    Let $c=\frac{(d-2)^2}{4}$, there exists $C>0$, so that
    \begin{equation}
        \begin{aligned}
            &\frac{C}{|x|^2}\geq V(x) \geq -\frac{c}{|x|^2}\\ 
            &-\partial_r(|x|V(x))\geq -\frac{c}{|x|^2}
        \end{aligned}
    \end{equation}
    This is easily verified.
\end{proof}

Now, we follow the strategy of \cite{rodnianski2001time} to get the folowing estimates. For the convenience of the reader, we write the proof in the following proposition:
\begin{prop}[Strichartz Estimates for Perturbed Klein Gordan Flow on $\mathbb{R}^d$]
    Suppose u(t,x) solves (\ref{eq for Strichartz}) on $I\times \mathbb{R}^d$ with $(u_0,u_1)\in H^\frac{1}{2}\times H^{-\frac{1}{2}}$ and are radially symmetric, For admissible pairs $(p,q,d,\theta,\frac{1}{2})$ and $(\tilde{p},\Tilde{q},d,\theta,\frac{1}{2})$, we have
    \begin{equation}
        \|u\|_{L^p_tL^q_x} + \|(u,u_t)\|_{C^0_tH^\frac{1}{2}\times H^{-\frac{1}{2}}} \lesssim \|(u_0,u_1)\|_{H^\frac{1}{2}\times H^{-\frac{1}{2}}} + \|F\|_{L^{\Tilde{p}'}_tL^{\Tilde{q}'}_x}
    \end{equation}
\end{prop}

\begin{proof}
    Let $H_0=\begin{bmatrix}
        0 & -1\\
        \Delta & 0
    \end{bmatrix}$
    and $H=\begin{bmatrix}
        0 & -1\\
        \Delta-1+V & 0
    \end{bmatrix}$.\\
    Note that $\begin{bmatrix}
        f\\
        g
    \end{bmatrix}\rightarrow \begin{bmatrix}
        u\\
        u_t
    \end{bmatrix}=e^{tH_0}\begin{bmatrix}
        f\\
        g
    \end{bmatrix}$ is the linear Klein Gordan without perturbation, thus satisfies
    \begin{equation}
        \|I_1e^{tH_0}\begin{bmatrix}
            f\\
            g
        \end{bmatrix}\|_{L^p_tL^q_x} \lesssim \|(f,g)\|_{H^\frac{1}{2}\times H^{-\frac{1}{2}}}
    \end{equation}
    for admissible pair $(p,q,d,\theta,\frac{1}{2})$ and $I_1=\begin{bmatrix}
            1 & 0
        \end{bmatrix}$.\\
    Denote $A=\begin{bmatrix}
        0 & 0\\
        \sqrt{|V|} & 0 
    \end{bmatrix}$ and $B=\begin{bmatrix}
        -\sqrt{|V|} & 0\\
        0 & 0
    \end{bmatrix}$. Thus $H-H_0=AB$
    Duhamel's formula yields
    \begin{equation}\label{Duhamel operator}
        e^{tH}\begin{bmatrix}
            f\\
            g
        \end{bmatrix}=e^{tH_0}\begin{bmatrix}
            f\\
            g
        \end{bmatrix} - \int_0^t e^{(t-s)H_0}ABe^{isH}\begin{bmatrix}
            f\\
            g
        \end{bmatrix}ds
    \end{equation}
    To prove that $\|I_1e^{tH}\begin{bmatrix}
        f\\
        g
    \end{bmatrix}\|_{L^p_tL^q_x}\lesssim \|(f,g)\|_{H^\frac{1}{2}\times H^{-\frac{1}{2}}}$, the first term in (\ref{Duhamel operator}) automatically satisfies the estimate due to unperturbed Strichartz estimate. It only remains to prove the second term also maps $(f,g)\in H^\frac{1}{2}\times H^{-\frac{1}{2}}$ to $L^p_tL^q_x$. In fact, if we write the second term in the following form
    \begin{equation*}
        \int_0^t e^{(t-s)H_0}AB\begin{bmatrix}
            f\\
            g
        \end{bmatrix}ds=T_1 T_2 \begin{bmatrix}
            f\\
            g
        \end{bmatrix}
    \end{equation*}
    where 
    \begin{equation}
        T_2\begin{bmatrix}
            f\\
            g
        \end{bmatrix}=\begin{bmatrix}
            -\sqrt{|V|} & 0\\
            0 & 0
        \end{bmatrix}e^{isH}\begin{bmatrix}
            f\\
            g
        \end{bmatrix}=\begin{bmatrix}
            -\sqrt{|V|}\left(cos\sqrt{-\Delta+\rho^2+V}f + \frac{sin\sqrt{-\Delta+\rho^2+V}}{\sqrt{-\Delta+\rho^2+V}}g\right)\\
            0
        \end{bmatrix}
    \end{equation}
    and
    \begin{equation}
        T_1\begin{bmatrix}
            F\\
            G
        \end{bmatrix} = I_1\int_0^t e^{(t-s)H_0}\begin{bmatrix}
            0 & 0\\
            \sqrt{|V|} & 0
        \end{bmatrix}\begin{bmatrix}
            F\\
            G
        \end{bmatrix}ds
    \end{equation}

    By (\ref{Kato smoothing}), we have $T_2: H^\frac{1}{2}\times H^{-\frac{1}{2}} \rightarrow L^2_{t,x}\times 0$. Thus it suffices to prove that $T_1: L^2_{t,x}\times 0 \rightarrow L^p_tL^q_x$
    By Christ Kiselev Lemma, it suffices to show the mapping property for 
    \begin{equation}\label{Strichartz T_1}
        \Tilde{T_1}=I_1e^{tH_0}\int_0^\infty e^{-sH_0}\begin{bmatrix}
            0 & 0\\
            \sqrt{|V|} & 0
        \end{bmatrix}ds
    \end{equation}
    By Lemma \ref{Lemma Strichartz Standard}, $I_1e^{tH_0}$ maps $H^\frac{1}{2}\times H^{-\frac{1}{2}}$ to $L^p_tL^q_x$. Thus in order to prove (\ref{Strichartz T_1}), it suffices to show that $T_3=\int_0^\infty e^{-sH_0}\begin{bmatrix}
            0 & 0\\
            \sqrt{|V|} & 0
        \end{bmatrix}ds$ maps $L^2_{t,x}\times 0$ to $H^\frac{1}{2}\times H^{-\frac{1}{2}}$. This can be seen by duality: given any $(\phi,\psi)\in H^{-\frac{1}{2}}\times H^\frac{1}{2}$, we have

    \begin{equation}
        \begin{aligned}
            <T_3\begin{bmatrix}
                F\\
                G
            \end{bmatrix},\begin{bmatrix}
                \phi\\
                \psi
            \end{bmatrix}>_{L^2_{x}}&= \int_0^\infty F(s)\cdot \sqrt{|V|}\left( \frac{sin\sqrt{-\Delta+\rho^2}}{\sqrt{-\Delta+\rho^2}}\phi + cos\sqrt{-\Delta+\rho^2}\psi \right)\\
            &\leq \|F\|_{L^2_{t,x}} \|\sqrt{|V|}\left( \frac{sins\sqrt{-\Delta+\rho^2}}{\sqrt{-\Delta+\rho^2}}\phi + coss\sqrt{-\Delta+\rho^2}\psi \right)\|_{L^2_{t,x}}\\
            &\leq \|F\|_{L^2_{t,x}}\|(\phi,\psi)\|_{H^{-\frac{1}{2}}\times H^\frac{1}{2}}
        \end{aligned}
    \end{equation}

    This proves
    \begin{equation}
        \|e^{t\sqrt{-\Delta+\rho^2+V}}f\|_{L^p_tL^q_x} \lesssim \|f\|_{H^\frac{1}{2}}
    \end{equation}
    The nonlinear estimate follows from duality. This completes our proof of the perturbed Strichartz estimate on $\mathbb{R}^d$.

\end{proof}

Now, we revert the coordinates back to $\hs$ to get the Strichartz estimate on $\hs$
\begin{cor}[Strichartz Estimates for Wave Flow on $\hs$]
    Suppose u(t,x) solves (\ref{eq for Strichartz}) on $I\times\hs$ with $(u_0,u_1)\in H^\frac{1}{2}\times H^{-\frac{1}{2}}$ and are radially symmetric. Suppose $(p,q,d,\theta,\frac{1}{2})$ and $(\Tilde{p},\Tilde{q},d,\theta,\frac{1}{2})$ are admissible in the same sense as on $\mathbb{R}^d$, then we have
    \begin{equation}
        \|w_q u\|_{L^p_tL^q_x} + \|u\|_{C^0_tH^\frac{1}{2}\times H^{-\frac{1}{2}}} \lesssim \|(u_0,u_1)\|_{H^\frac{1}{2}\times H^{-\frac{1}{2}}} + \|w_{\Tilde{q}'}F\|_{L^{\Tilde{p}'}_tL^{\Tilde{q}'}_x}
    \end{equation}
    where
    \begin{equation}
        w_q=\left(\frac{sinhr}{r}\right)^{(1-\frac{2}{q})\rho}
    \end{equation}
\end{cor}

Global wellposedness and scattering of the solution to (\ref{eq:nlw}) for small data follows from Strichartz estimates.
\begin{cor}[Small Data Result for Conformal Wave Equation]
    There exists $\delta_0>0$, such that the solution to (\ref{eq:nlw}) is globally wellposed and scatters if $\|(u_0,u_1)\|_{H^\frac{1}{2}\times H^{-\frac{1}{2}}}\leq \delta_0$.
\end{cor}

\section{Morawetz Estimates}
In this section, we prove several Morawetz inequalities derived from choosing different potentials. Given smooth function $a(x)\in C^\infty(\mathbb{H}^d)$, let u be a radial solution to the modified NLW (\ref{nlw error}) with error term $\mathcal{N}$
\begin{equation}\label{nlw error}
    u_{tt}-\Delta_{\mathbb{H}^d}u+u^p=\mathcal{N}
\end{equation}
For suitably defined functions $\phi, a \in C^\infty(\hs-\{0\})$ on $\hs$, define the Morawetz potential to be
\begin{equation}\label{eq MP}
    M(t)=-\int_{\hs}\phi\left(u_t\nabla u\cdot\nabla a+u_tu\frac{\Delta a}{2}\right)
\end{equation}

In order to derive control from the Morawetz potential, we will have to impose certain restrictions on the choices of its component functions. 
\begin{asp}\label{asp a}
    $a \in C^\infty (\hs - \{0\})$ is radially symmetric and satisfies
    \begin{equation}\label{asp a}
        |\nabla a| \leq C, \hspace{2em}
        \Delta a\geq 0, \hspace{2em} \Delta^2 a\leq 0
    \end{equation}
    and $D^2 a \geq 0 $ for any tangent vector v along radial direction.
\end{asp}
and 
\begin{asp}\label{asp phi}
    $\phi \in C^{\infty}(\hs-\{0\})$ is radially symmetric and satisfies
    \begin{equation} \label{asp phi 1}
        \left\{\begin{aligned}
            & \partial_r \phi \geq 0 \\
            & |\nabla \phi| \leq \frac{C}{r}, \hspace{2em} |\Delta \phi| \leq \frac{C}{r^2}
        \end{aligned}\right.
    \end{equation}
\end{asp}
We first compute the time derivative of the Morawetz potential.
\begin{lem}[Generalized Morawetz Potential]
    Suppose u is a radial symmetric solution to (\ref{nlw error}) on the time interval I, then 
    \begin{equation}\label{eq MP dM/dt}
        \begin{aligned}
             \frac{dM}{dt}&=\int_{\hs}a^{''}(r)|\nabla u|^2 + \frac{1}{2} \int (\nabla \phi \cdot \nabla a) |\nabla u|^2 + \frac{1}{2}(\nabla \phi \cdot \nabla a) u_t^2 \\
                         &+ \int_{\hs}(\frac{1}{2}-\frac{1}{p+1}) \phi \Delta a u^{p+1} - \frac{1}{p+1} (\nabla \phi \cdot \nabla a) u^{p+1} \\
                         & - \int_{\hs}\frac{1}{4}(\phi \Delta^2 a + 2\nabla \phi \cdot \nabla \Delta a + \Delta \phi \Delta a) u^2\\
                         &+ \int_{\hs}\phi\mathcal{N}(\nabla a \cdot \nabla u) + \frac{1}{2}\phi\mathcal{N}u\Delta a
        \end{aligned}
    \end{equation}
\end{lem}

\begin{proof}
    We prove (1) for $u\in C^\infty(\hs) \cap H^1(\hs)$ and $u_t \in C^\infty(\hs) \cap L^2(\hs)$. The general inequality case follows from a density argument.
    Differentiating $M(t)$ in t, we obtain
    \begin{equation}\label{eq dM/dt 1}
        \begin{aligned}
            \frac{dM}{dt} & = -\int_\hs \phi (u_{tt} \nabla u \cdot \nabla a + u_t \nabla u_t \cdot \nabla a + \frac{1}{2}u_{tt} u \Delta a + \frac{1}{2} u_t^2 \Delta a ) \\
            & = I + II + III + IV
        \end{aligned}
    \end{equation}
    where I~IV correspond to the four terms in (\ref{eq dM/dt 1}) in order.
    Using (\ref{nlw error}) and integrating in parts, we get
    \begin{equation}\label{eq MP I1}
        \begin{aligned}
            I_1 & := -\int \phi \Delta u \nabla u \cdot \nabla a \\
            & = \int \phi D^2 a (\nabla u, \nabla u) + \frac{1}{2}\int \phi \nabla a \cdot \nabla |\nabla u|^2 + \int (\nabla (\psi_\epsilon \phi) \cdot \nabla u)(\nabla a \cdot \nabla u) \\
            & = \int \phi D^2 a (\nabla u, \nabla u) - \frac{1}{2}\int \phi \Delta a |\nabla u|^2 + \frac{1}{2}\int (\nabla \phi \cdot \nabla a)|\nabla u|^2 + \frac{1}{2}\int \phi (\nabla \cdot \nabla a)|\nabla u|^2 
        \end{aligned}
    \end{equation}
    and 
    \begin{equation}
        \begin{aligned}
            I_2 & := \int \phi u^p \nabla u \cdot \nabla a \\
            & = -\frac{1}{p+1}\int \phi u^{p+1} \Delta a - \frac{1}{p+1}\int (\nabla \phi \cdot \nabla a) u^{p+1}
        \end{aligned}
    \end{equation}
    \begin{equation}
        II = \frac{1}{2}\int \Delta a u_t^2 + \frac{1}{2} \int (\nabla \phi \cdot \nabla a) u_t^2 
    \end{equation}
    \begin{equation}
        \begin{aligned}
            III_1 & := -\frac{1}{2}\int \phi u\Delta u  \Delta a \\
            & = \frac{1}{2}\int \nabla \phi \cdot u\nabla u \Delta a + \frac{1}{2} \int \phi |\nabla u|^2 \Delta a + \frac{1}{2} \int \phi u\nabla u \cdot \nabla \Delta a \\
            & = -\frac{1}{4} \int ( \Delta \phi \Delta a + 2 \nabla \phi \cdot \nabla \Delta a +  \phi \Delta^2 a ) u^2 + \frac{1}{2} \int \phi \Delta a |\nabla u|^2
        \end{aligned}
    \end{equation}
    \begin{equation}
        III_2 = \frac{1}{2} \int \phi u^{p+1} \Delta a  
    \end{equation}
    \begin{equation}\label{eq MP IV}
        IV = - \frac{1}{2} \int \phi \Delta a u_t^2 
    \end{equation}
    Collecting (\ref{eq MP I1})-(\ref{eq MP IV}), we have
    \begin{equation}
        \begin{aligned}
            \frac{dM}{dt} & = \int \phi D^2 a (\nabla u, \nabla u) + \frac{1}{2} \int (\nabla \phi \cdot \nabla a) |\nabla u|^2 + \frac{1}{2} \int (\nabla \phi \cdot \nabla a) u_t^2 \\
            & + (\frac{1}{2} - \frac{1}{p+1}) \int \phi \Delta a u^{p+1} 
            - \frac{1}{4} \int ( \phi \Delta^2 a + 2\nabla \phi \cdot \nabla \Delta a + \Delta \phi ) u^2 \\
            & - \int \phi \mathcal{N} \nabla u \cdot \nabla a - \frac{1}{2} \int \phi \Delta a \mathcal{N} u 
        \end{aligned}
    \end{equation}
    Since u is radial, we have
    \begin{equation*}
        D^2 a (\nabla u, \nabla u) = \partial_{r}^2 a(r) |\nabla u|^2
    \end{equation*}
    This proves the lemma.
\end{proof}

Now we choose suitable functions $a_j$ and $\phi_j \in C^\infty(\hs - \{0\})$ which yields useful spacetime bounds.

\begin{lem}\label{M1}
    Given $0 < \delta \leq 1$, there exists function $a \in C^\infty(\hs-\{0\})$ and positive constants $C = C(d)$ and $R = R(d)$ such that
        \begin{equation}\label{eq MP a 1}
            \Delta a = \left\{\begin{aligned}
                & \frac{1}{r^{1-\delta}}, \hspace{2em} 0 < r \leq 1 \\
                & 1, \hspace{2em} r > 1
            \end{aligned}\right.
            \hspace{2em} , \hspace{2em}
            |\nabla a| \leq \left\{\begin{aligned}
                & C r^{\delta}, \hspace{2em} 0 < r \leq R \\
                & C, \hspace{2em} r > R 
            \end{aligned}\right.
        \end{equation}
        and
        \begin{equation}\label{eq MP a 2}
            \partial_{r}^2 a(r) \geq \left\{\begin{aligned}
                & \frac{1}{C}r^{-(1-\delta)}, \hspace{2em} 0 < r \leq R \\
                & \frac{1}{C}e^{-2r}, \hspace{2em} r > R
            \end{aligned}\right.
        \end{equation}
\end{lem}

\begin{proof}
    The case where $\delta = 1$ has been proved in \cite{ionescu2009semilinear} except for the estimate (\ref{eq MP a 2}). For the readers' convenience, we recall that $a_1$ is obtained by solving the following partial differential equation
    \begin{equation}\label{eq a1 def}
        \partial_{r}^2 a_1 + \frac{(d-1)cosh r}{sinh r}\partial_r a_1 = 1
    \end{equation}
    which yields
    \begin{equation}
        \partial_r a_1 = \frac{1}{sinh^{d-1}r} \int_0^r sinh^{d-1}s ds
    \end{equation}
    $a_1$ is obtained by integrating again with the initial value $a_1(0) = 0$. $a_1$ has been shown to satisfy Assumption \ref{asp a}. See \cite{ionescu2009semilinear} Lemma 4.2 for details. Now we prove the lower bound (\ref{eq MP a 2}), which was not shown by the referenced work. Using $\Delta a = 1$, we have
    \begin{equation}
        \partial_{r}^2 a = 1 - \frac{(d-1)coshr}{sinhr} \cdot \frac{1}{sinh^{d-1}r} \int_0^r sinh^{d-1}s ds 
    \end{equation}
    Denote $I_d(r) = \partial_{r}^2 a^{(d)}(r)$ where $a^{(d)}$ is defined by (\ref{eq a1 def}) in dimension d. Using the inductive equation
    \begin{equation*}
        \int_0^r sinh^d s ds = -\frac{d-2}{d-1} \int_0^r sinh^{d-2} s ds + \frac{1}{d-1} sinh^{d-2} r cosh r
    \end{equation*}
    , we get
    \begin{equation}
        I_d(r) = -\frac{d-2}{d-3} \cdot \frac{1}{sinh^{2}r} I_{d-2}(r) + \frac{1}{d-3} \cdot \frac{1}{sinh^2 r}
    \end{equation}
    Also, we compute $I_3(r)$ and $I_4(r)$ explicitly as below:
    \begin{equation}
        \begin{aligned}
            & I_3(r) = \frac{rcosh r - sinh r}{sinh^3 r} \\
            & I_4(r) = \frac{(cosh r - 1)^2}{sinh^4 r} 
        \end{aligned}
    \end{equation}
    By induction, we can show that the Taylor expansion of $I_d(r)$ at $r=0$ writes as
    \begin{equation}
        I_d(r) = \frac{1}{d} + \frac{d-1}{d(d+2)} r^2 + o(r^2)
    \end{equation}
    and that 
    \begin{equation}
        I_d(r) \geq \frac{1}{C}e^{-2r} 
    \end{equation}
    for some $C=C(d)>0$ as $r \rightarrow +\infty$.
    This proves the latter inequality of (\ref{eq MP a 2}).

    For $0 < \delta < 1$, a is defined by
    \begin{equation}\label{eq a2 def}
        \partial_{r}^2 a + \frac{(d-1)cosh r}{sinh r}\partial_r a = \left\{\begin{aligned}
            & \frac{1}{r^{1-\delta}}, \hspace{2em} 0 < r \leq 1 \\
            & 1, \hspace{2em} r > 1
        \end{aligned}\right.
    \end{equation}
    which yields
    \begin{equation}
        \partial_r a = \frac{1}{sinh^{d-1}r} \int_0^r \frac{sinh^{d-1}s}{s^{1-\delta}} ds
    \end{equation}
    for $0 < r \leq R$. (\ref{eq MP a 1}) follows directly. To see (\ref{eq MP a 2}), note that
    \begin{equation}
        \begin{aligned}
            \partial_{r}^2 a = \frac{1}{r^{1-\delta}} - \frac{(d-1)coshr}{sinhr} \cdot \frac{1}{sinh^{d-1}r} \int_0^r \frac{sinh^{d-1}s}{s^{1-\delta}} ds
        \end{aligned}
    \end{equation}
    Using Taylor expansion near $r = 0$, we get
    \begin{equation}\label{eq a2 Taylor}
        \partial_{r}^2 a = \frac{\delta}{d-1+\delta} \cdot r^{-(1-\delta)} + o(r^{-(1-\delta)})
    \end{equation}
    This proves the $r < R$ part of (\ref{eq MP a 2}) for $0 < \delta < 1$. The $r > R$ part can be obtained by the same argument for $\delta = 1$ with a slight modification. This concludes the proof of the lemma.
   
\end{proof}

Define $a_1, a_2 \in C^\infty(\hs-\{0\})$ as in the preceding lemma with $\delta = 1$ and some small $\delta > 0$ respectively. Define Morawetz potential $M_1$ and $M_2$ accordingly by (\ref{eq MP}). We have the following lower bound for their growth rate.

\begin{cor}\label{cor MP 1,2}
    There exists constants $C = C(d)>0$ and $R = R(d)>0$ such that the following estimates holds:
        \begin{equation}\label{MP1}
            \frac{dM_1}{dt} \geq \frac{1}{C}\|e^{-r}\nabla u\|_2^2 + \frac{1}{C}\|u\|_{p+1}^{p+1} - C\|\mathcal{N}\nabla u\|_1 - C\|\mathcal{N}u\|_1
        \end{equation}
    and
        \begin{equation}\label{MP2}
            \begin{aligned}
                \frac{dM_2}{dt} & \geq \frac{1}{C}\|r^{\frac{-(1-\delta)}{2}}\nabla u\|_{L^2_x(r<R)}^2 
                + \frac{1}{C}\|r^{-\frac{1}{p+1}(1-\delta) } u\|_{L^{p+1}_x(r<R)}^{p+1} + \frac{1}{C}\|r^{\frac{-(3-\delta)}{2}}u\|_{L^2_x(r<R)}^2 \\
                & - C\|\mathcal{N}\nabla u\|_1 - C\|\frac{\mathcal{N}u}{r^{1-\delta}}\|_1 - C\|\mathcal{N}u\|_1
            \end{aligned}
        \end{equation}
\end{cor}

\begin{proof}
    These follow directly from the preceding lemma and (\ref{eq MP dM/dt}).
\end{proof}

Note that integrating (\ref{MP1}) and (\ref{MP2}) gives us the familiar Morawetz estimate of the form
\begin{equation*}
    \|e^{-r} \nabla u\|_{L^2_{t,x}}^2 + \|u\|_{p+1}^{p+1} \lesssim \sup_{t\in I} |M_1(t)| + \|\mathcal{N}\nabla u\|_1 + \|\mathcal{N}u\|_1
\end{equation*}
However, neither (\ref{MP1}) and (\ref{MP2}) gives us any control on $v_t$. This is because when $\phi = 1$, the $v_t$ terms in (\ref{eq MP dM/dt}) cancels. To counter that, we define the following two modified Morawetz potential, by introducing non-constant potential functions $\phi$. The purpose of $M_3$ is to yield control for $v_t$ away from origin and $M_4$ for near origin.

Let $\phi_3$ be defined by
\begin{equation}
    \phi_3(r) = 1 - e^{-2r}
\end{equation}
and $\phi_4$ by
\begin{equation}
    \phi_4(r) = log(r)\chi(r)
\end{equation}
where $\chi(r)$ is a smooth bump function supported in $B(0, 1)$ and equal to 1 in $B(0, \frac{1}{2})$.
Define modified Morawetz potential $M_3$ by (\ref{eq MP}) with $a_1$ and $\phi_3$. And similarly $M_4$ is defined with $a_2$ and $\phi_4$. We have the following lower bound for their growth rate.

\begin{cor}\label{cor MP 3,4}
    There exists positive constants $C = C(d)$ and $R=R(d)$ such that
    \begin{equation}
        \begin{aligned}
            \frac{dM_3}{dt} \geq \frac{1}{C}\|e^{-r}v_t\|_{L^2_x(r>R)}^2 - C\|u\|_{p+1}^{p+1} - C\|q(x)u\|_2^2 - C\|\mathcal{N}\nabla u\|_1 - C\|\mathcal{N}u\|_1
        \end{aligned}
    \end{equation}
    where 
    \begin{equation*}
        q(x)=\left\{\begin{aligned}
        &r^{-\frac{1}{2}},\ \ r\leq R\\
        &e^{-r}, \ \ r\geq R
    \end{aligned}\right.
    \end{equation*}.
    For $M_4$, we have that for some $0 < \delta' < \delta$:
    \begin{equation}
        \begin{aligned}
            \frac{dM_4}{dt} & \geq \frac{1}{C}\|r^{-\frac{1-\delta}{2}} u_t\|_{L^2_x(r<R)}^2 - C\|r^{-\frac{1-\delta}{2}}\nabla u\|_{L^2_x(r<R)}^2 - C\|r^{-\frac{1}{p+1}(1-\delta)}u\|_{L^{p+1}_x(r<R)}^{p+1} \\
            & - C\|r^{-\frac{3-\delta'}{2}}u\|_{L^2_x(r<R)}^2 - C\|\mathcal{N}\nabla u\|_1 - C\|\mathcal{N}\frac{u}{r}\|_1
        \end{aligned}
    \end{equation}
\end{cor}

\begin{proof}
    The positive $u_t$ term comes from $\int (\nabla a \cdot \nabla \phi) u_t^2$ in (\ref{eq MP dM/dt}). Note that $\nabla a \cdot \nabla \phi > 0$ in $M_3$ and $M_4$ which guarantees positiveness of the $u_t$ term. In addition, we have that for fixed $0 < \delta' < \delta$, the following inequality holds for $r \leq 1$ for some $C>0$:
    \begin{equation*}
        | logr \cdot r^{\delta} | \leq C r^{\delta'}
    \end{equation*}
    Thus we have
    \begin{equation}
        \begin{aligned}
            & | \nabla \phi \cdot \nabla a | \leq C r^{-(1-\delta)} \\
            & |\phi \Delta^2 a| \leq C r^{-(3-\delta')}
        \end{aligned}
    \end{equation}
    Plugging these into (\ref{eq MP dM/dt}) proves the estimates. This ensures introducing $\phi$ to control $u_t^2$ does not break our control over other terms like $|\nabla u|^2$ and $u^2$. 
\end{proof}

\section{Global Well-posedness and Scattering when $3\leq d \leq 5$}
Since the initial data (\ref{IVP}) lies below $H^1$, we are left with no energy or other known conserved quantity that controls $H^\frac{1}{2} \times H^{-\frac{1}{2}}$ norm. To counter this difficulty, we smooth up the initial data and evolve it under a perturbed NLW equation in a way such that the global wellposedness of the original solution u is equivalent to that of the perturbed solution. 
In the following Proposition, we state the mechanism of Fourier truncation method.
\begin{prop}
    Given initial value problem (\ref{eq:nlw}),(\ref{IVP}), we split the initial data into a high frequency small piece and a smoother large piece, in the following sense: for some $s>0$,
    \begin{equation}
        (u_0,u_1)=(\omega_0,\omega_1)+(v_0,v_1)
    \end{equation}
    where $(\omega_0,\omega_1)$ is the high frequency piece defined by
    \begin{equation}
        (\omega_0,\omega_1)=(P_{<s}u_0,P_{<s}u_1)
    \end{equation}
    and $(v_0,v_1)$ is the low frequency piece defined by
    \begin{equation}
        (v_0,v_1)=(P_{\geq s}u_0,P_{\geq s}u_1)
    \end{equation}
    Moreover, for any $0<\epsilon_1<\epsilon$, v and $\omega$ satisfy
    \begin{equation}\label{IV truc}
        \begin{aligned}
            &\|(\omega_0,\omega_1)\|_{H^{\frac{1}{2}+\epsilon_1} \times H^{-\frac{1}{2}+\epsilon_1}} \lesssim s^{\frac{1}{2}(\epsilon-\epsilon_1)}\|(u_0,u_1)\|_{H^{\frac{1}{2}+\epsilon} \times H^{-\frac{1}{2}+\epsilon}}\\
            &\|(v_0,v_1)\|_{H^1 \times L^2} \lesssim s^{\frac{1}{2}(\frac{1}{2}-\epsilon)}\|(u_0,u_1)\|_{H^{\frac{1}{2}+\epsilon} \times H^{-\frac{1}{2}+\epsilon}}
        \end{aligned}
    \end{equation}
\end{prop}
(\ref{IV truc}) comes from Bernstein inequality. If we choose s small enough depending on \\
$\|(u_0,u_1)\|_{H^{\frac{1}{2}+\epsilon} \times H^{-\frac{1}{2}+\epsilon}}$, then $\|(\omega_0,\omega_1)\|_{H^{\frac{1}{2}+\epsilon_1} \times H^{-\frac{1}{2}+\epsilon_1}}<<1$. Thus by small data result, the evolution of $(\omega_0,\omega_1)$ under the equation (\ref{eq:nlw}) exists globally in time.

\begin{cor}\label{omega estimate}
    Let $(\omega,\omega_t)$ be the solution to (\ref{eq:nlw}) with initial data $(\omega_0,\omega_1)$. For any $\delta_0>0$, there exists $s>0$ small enough, such that $(\omega,\omega_t)$ exists globally in time and satisfies the following bound
    \begin{equation}
        \begin{aligned}
            &\sup_t\|(\omega,\omega_t)\|_{H^{\frac{1}{2}+\frac{\epsilon}{2}} \times H^{-\frac{1}{2}+\frac{\epsilon}{2}}}<\delta_0 \\
            &\|\omega\|_{\hnorm{\frac{2(d+1)}{d-1}}} < \delta_0
        \end{aligned}
    \end{equation}
    Moreoever, there exists $\delta_\epsilon$ which depends on $\epsilon$, such that
    \begin{equation}
        \|r^{\frac{d-1}{2}-\delta_\epsilon}\omega\|_{L^\infty_{t,x}(I\times\{r<R\}))}+\|e^r\omega\|_{L^\frac{2(d+1)}{d-3}_{t,x}(I\times\{r>R\})}+\|e^{\frac{d-1}{4}r}\omega\|_{L^{\frac{4(d+1)}{d-1}}_{t,x}(I\times\{r>R\})}\lesssim \delta_0
    \end{equation}
\end{cor}
\begin{proof}
    It only remains to prove the last inequality, which follows from the interpolation between
    \begin{equation*}
        \begin{aligned}
            \|sinh^\frac{d-1}{2}r\omega\|_{L^\infty_{t,x}}&\lesssim \sup_t\|\omega\|_{H^{\frac{1}{2}+\frac{\epsilon}{2}}}\\
            \|\omega\|_{L^\infty_{t,x}}&\lesssim \sup_t\|\omega\|_{H^{\frac{d}{2}+\frac{\epsilon}{2}}}\\
            \|\omega\|_{L^\frac{2(d+1)}{d-1}_{t,x}}&\lesssim \|(\omega_0,\omega_1)\|_{H^\frac{1}{2}\times H^{-\frac{1}{2}}}
        \end{aligned}
    \end{equation*}
\end{proof}

Now, suppose u(t,x) solves equation (\ref{eq:nlw}) on time interval I. We cut the high frequency evolution $\omega$ from u and denote $v=u-\omega$. Let $E(t)$ be the energy of v for $t\in I$. By (\ref{IV truc}), we have\\

\begin{equation}
    E(0)\lesssim s^{\delta-\frac{1}{2}}\|(u_0,u_1)\|_{H^{\frac{1}{2}+\delta}\times H^{-\frac{1}{2}+\delta}}\lesssim_{\|(u_0,u_1)\|_{H^{\frac{1}{2}+\delta}\times H^{-\frac{1}{2}+\delta}}} 1
\end{equation}
Our goal now is to show that the energy of v stays bounded throughout I. Suppose that is true, then by local well-posedness theory, v can be extended beyond I and global well-posedness follows. v solves the following equation on I:
\begin{equation}
    \begin{aligned}
        &v_{tt}-\Delta v+v^\frac{d+3}{d-1}=\mathcal{N}\\
        &v(0)=v_0,\ \ v_t(0)=v_1
    \end{aligned}
\end{equation}
where $\mathcal{N}=-u^\frac{d+3}{d-1}+\omega^\frac{d+3}{d-1}+v^\frac{d+3}{d-1}$\\
We will consider the cases where $3\leq d \leq 5$ and $d>5$ separately. This is because in higher dimension the nonlinearity has lower power, which requires different treatment compared to lower dimensional case.


\subsection{Case I: $3\leq d \leq 5$}
In this case, the error term $\mathcal{N}\lesssim |v^\frac{d-1}{4}\omega|+|v\omega^\frac{4}{d-1}|$.
And the derivative of $E(t)$ writes as
\begin{equation}
    \frac{dE(t)}{dt}=<\mathcal{N},v_t>\lesssim \left|\int v^\frac{4}{d-1}\omega v_t\right|+\left|\int v\omega^\frac{4}{d-1} v_t \right|
\end{equation}

For heuristic, let us look at when $d=3$ and the second term in $\mathcal{N}=\left|\int v\omega^2 v_t \right|$. We can estimate it by 
\begin{equation*}
    \begin{aligned}
        \left|\int v\omega^2 v_t \right|&\lesssim \|v\|_6\|v_t\|_2\|\omega\|_6\\
        &\lesssim \|\omega\|_6^2E(t)
    \end{aligned}
\end{equation*}
If we temporarily overlook the presence of the first term, the Gr\"{o}nwall inequality would give us $E(t)\lesssim e^{\|\omega\|_{\norm{2}{6}}}\lesssim 1$. However, the term $\left|\int v^2\omega v_t\right|$ cannot be estimated in the same fashion, for it yields $E^{\frac{3}{2}}(t)$ which is not enough for Gronwall inequality to rule out finite time blowup. For this purpose, we introduce the following modified energy, which takes advantage of the Morawetz potential defined in Section 3.

\begin{defn}[Modified Energy]
    Under the previous settings, define in terms of v, 
    \begin{equation}
        \mathcal{E}(t)=E(t)-c_1M_1(t)-c_2M_2(t)-c_3M_3(t)-c_4M_4(t)
    \end{equation}
    where $M_j$ are defined in section 3, with $0<\delta<1$ to be determined for $M_2$ and $M_4$. $c_j>0$ are small constants to be determined.
\end{defn}


\begin{prop}\label{Core Estimate}
    For any $0<\delta<1$, there exist $c_j>0$ and 
    \begin{enumerate}
        \item $\mathcal{E}(t) \sim E(t)$ for $t\in I$
        \item There exists $R>0$, such that the time derivative of $\mathcal{E}$ is bounded by
        \begin{equation}\label{eq energy growth rate}
            \frac{d\mathcal{E}}{dt}\lesssim \|\omega\|_\frac{2(d+1)}{d-1}^\frac{2(d+1)}{d-1} + \|e^r\omega\|_{L^\frac{2(d+1)}{d-3}_x(r>R)}^\frac{2(d+1)}{d-3} + \|e^{\frac{d-1}{4}r}\omega\|_{L^\frac{4(d+1)}{d-1}_x(r>R)}^\frac{4(d+1)}{d-1}
        \end{equation}
    \end{enumerate}
    
\end{prop}

Let us postpone the proof of Proposition \ref{Core Estimate} and continue on to the proof of Theorem \ref{Main Thm}. We will see that Theorem \ref{Main Thm} is a direct result of Proposition \ref{Core Estimate}. \\

\emph{Proof of the theorem \ref{Main Thm}:} Integrating (\ref{eq energy growth rate}) in time will yield
\begin{equation}
    \sup_t\mathcal{E}(t) \lesssim \|\omega\|_{L^\frac{2(d+1)}{d-1}_{t,x}}^\frac{2(d+1)}{d-1} + \|e^r \omega\|_{\hnorm{\frac{2(d+1)}{d-3}}(I\times \{r>R\})}^\frac{2(d+1)}{d-3} + \|e^{\frac{d-1}{4}r}\omega\|_{\hnorm{\frac{4(d+1)}{d-1}}(I\times \{r>R\})}
\end{equation}
Recall from Corallary \ref{omega estimate} that
\begin{equation}
    \sup_t \mathcal{E}(t) \lesssim \mathcal{E}(0) + \delta_0 
\end{equation}
Combining the preceding estimate with part (1) of \ref{Core Estimate} yields
\begin{equation}\label{Bounded Energy}
    \sup_{t\in I} E(t) \lesssim \sup_{t\in I} \mathcal{E}(t) \lesssim \mathcal{E}(0) + \delta_0 
\end{equation}
This proves global wellposedness. \\
Scattering follows directly from Morawetz estimate and (\ref{Bounded Energy})
\begin{equation}
    \|v\|_{\hnorm{\frac{2(d+1)}{d-1}}(I\times \hs)}^\frac{2(d+1)}{d-1} \lesssim \sup_{t\in I}E(t) \lesssim \mathcal{E}(0) + \delta_0 
\end{equation}
This concludes the proof of Theorem \ref{Main Thm}.

Now, we prove Proposition \ref{Core Estimate}.
\begin{proof}
    Part (1) of Proposition \ref{Core Estimate} easily follows from the boundedness of $|\nabla a|$ and $|\phi|$(except for $M_4$)
    \begin{equation*}
        |\int \phi(v_t\nabla a\cdot \nabla v)|\lesssim \|\phi\|_\infty\|\nabla a\|_\infty \|v_t\|_2\|\nabla v\|_2
    \end{equation*}
    By Hardy inequality,
    \begin{equation*}
        |\int \phi(v_t \frac{v}{r})|\lesssim \|\phi\|_\infty \|v_t\|_2\|\nabla v\|_2
    \end{equation*}
    Thus
    \begin{equation*}
        M_j(t)\lesssim E(t)
    \end{equation*}
    for j=1,2,3
    For $M_4$, 
    \begin{equation*}
        \begin{aligned}
            |M_4(t)|&\lesssim \int_{0<r\leq 1}| v_t\nabla v \cdot r^\delta logr|+|v_t\frac{v}{r}r^\delta logr|\\
            &\lesssim \|v_t\|_2\|\nabla v\|_2
        \end{aligned}
    \end{equation*}
    This proves part (1) of Proposition \ref{Core Estimate}, by choosing $c_j$ small enough.
    In order to prove part (2), we estimate $\frac{d\mathcal{E}}{dt}$ using results from Corollary \ref{cor MP 1,2} and \ref{cor MP 3,4}. In fact, we have

    
    \begin{equation}\label{growth rate}
    \begin{aligned}
        \frac{d\mathcal{E}(t)}{dt}&=\frac{dE(t)}{dt}-c_1\frac{dM_1(t)}{dt}-c_2\frac{dM_2(t)}{dt}-c_3\frac{dM_3(t)}{dt}-c_4\frac{M_4(t)}{dt}\\
        &\leq \|v^\frac{4}{d-1}\omega v_t\|_1 + \|v\omega^\frac{4}{d-1}v_t\|_1 \\
        & - c_1\left( \frac{1}{C}\|e^{-r}\nabla v\|_2 + \frac{1}{C}\|v\|_{\frac{2(d+1)}{d-1}}^{\frac{2(d+1)}{d-1}}- C\|\mathcal{N}\nabla v \|_1 - C\|\mathcal{N}v\|_1 \right)\\
        & - c_2\left( \frac{1}{C}\|r^{\frac{-1+\delta}{2}}\nabla v\|_{L^2_x(r<R)}^2 + \frac{1}{C}\|r^{\frac{-3+\delta}{2}}v\|_{L^2_x(r<R)}^2 - C\|\mathcal{N}\nabla v\|_1 - C\|\mathcal{N}\frac{v}{r^{1-\delta}}\|_1 \right)\\
        & - c_3\left( \frac{1}{C}\|e^{-r}v_t\|_{L^2_x(r>R)}^2 - C\|e^{-r}\nabla v\|_2^2 - C\|v\|_{\frac{2(d+1)}{d-1}}^{\frac{2(d+1)}{d-1}} - C\|q(x)u\|_2^2 - \|\mathcal{N}\nabla u\|_1 - C\|\mathcal{N}u\|_1\right)\\
        & - c_4\left( \frac{1}{C}\|r^{\frac{-1+\delta}{2}}\|_{L^2_x(r<R)}^2 - C\|r^{\frac{-1+\delta}{2}}\nabla u\|_{L^2_x(r<R)}^2 - C\|r^{-\frac{d-1}{2(d+1)}}u\|_{L^{\frac{2(d+1)}{d-1}}_x(r<R)}^{\frac{2(d+1)}{d-1}} - C\|r^{\frac{-3+\delta}{2}}u\|_{L^2_x(r<R)}^2 \right.\\
        & \ \ \ \ \ \left. - C\|\mathcal{N}\nabla u\|_1 - C\|\mathcal{N}\frac{u}{r}\|_1 \right)
    \end{aligned}
\end{equation}
    
Pick $c_3<<c_1$ and $c_4<<c_2$, then for some constant $c = c(d)>0$ we have

\begin{equation}\label{energy rate 2}
    \begin{aligned}
        \frac{d\mathcal{E}}{dt}&\leq \|v^\frac{4}{d-1}\omega v_t\|_1 + \|v\omega^\frac{4}{d-1}v_t\|_1 + C\|v^\frac{4}{d-1}\omega \nabla v\|_1 + \|v\omega^\frac{4}{d-1}\nabla v\|_1 + C\|v^\frac{4}{d-1}\omega v\|_1\\
        & + C\|v\omega^\frac{4}{d-1}v\|_1 + C\|v^\frac{4}{d-1}\omega\frac{v}{r}\|_1 + C\|v\omega^\frac{4}{d-1}\frac{v}{r}\|_1 \\
        & - c \|e^{-r}\nabla v\|_2^2 - c \|e^{-r}\nabla v\|_{L^2_x(r>R)}^2 -c\|r^\frac{-1+\delta}{2}v_t\|_{L^2_x(r<R)}^2 - c\|e^{-r}v_t\|_{L^2_x(r>R)}^2 \\
        & - c\|r^\frac{-3+\delta}{2}v\|_{L^2_x(r<R)}^2 - c\|v\|_{\frac{2(d+1)}{d-1}}^{\frac{2(d+1)}{d-1}} -c\|r^{-\frac{d-1}{2(d+1)}(1-\delta)}v\|_{L^\frac{2(d+1)}{d-1}_x(r<R)}^\frac{2(d+1)}{d-1}
    \end{aligned}
\end{equation}

The negative terms in (\ref{energy rate 2}) does not contribute to blow-up of the energy. We only need to bound the positive terms so that integrating (\ref{energy rate 2}) in time does not cause the RHS to approach infinity.\\

We will first deal with the $v_t$ terms, by H\"{o}lder inequality and Cauchy Schwartz inequality, we have
\begin{equation}\label{Term 1}
    \begin{aligned}
        \|v^\frac{4}{d-1}\omega v_t\|_{L^1_x(r<R)}&=\|r^{\frac{-1+\delta}{2}}v_t\cdot (r^{-3+\delta}v^2)^\alpha\cdot (r^{-(1-\delta)}v^\frac{2(d+1)}{d-1})^\beta\cdot (r^{\frac{d-1}{2}-\delta'}\omega)\|_{L^1_x(r<R)}\\
        &\lesssim \|r^\frac{-1+\delta}{2}v_t\|_{L^2_x(r<R)}\|r^{\frac{-3+\delta}{2}}v\|_{L^2_x(r<R)}^{2\alpha}\|r^{-\frac{d-1}{2(d+1)}(1-\delta)}v\|_{\frac{2(d+1)}{d-1}}^{\frac{2(d+1)}{d-1}\beta}\|r^{\frac{d-1}{2}-\delta}\omega\|_\infty\\
        &\lesssim \|r^{\frac{d-1}{2}-\delta}\omega\|_\infty \|r^\frac{-1+\delta}{2}v_t\|_{L^2_x(r<R)}^2  + \|r^{\frac{d-1}{2}-\delta}\omega\|_\infty \|r^{\frac{-3+\delta}{2}}v\|_{L^2_x(r<R)}^2 \\
        &+ \|r^{\frac{d-1}{2}-\delta}\omega\|_\infty \|r^{-\frac{d-1}{2(d+1)}(1-\delta)}v\|_{\frac{2(d+1)}{d-1}}^{\frac{2(d+1)}{d-1}}
    \end{aligned}
\end{equation}
where $\alpha = \frac{2}{d-1}-\frac{(d+1)(5-d)}{4(d-1)},\ \beta=\frac{5-d}{4}$ and $\delta'=(\frac{1}{2}+\alpha-\beta)\delta$. For convenience, we will abuse notation to write $\delta$ for any constant multiple of $\delta$, including $\delta'$.\\

By Corallary \ref{omega estimate}, $\|r^{\frac{d-1}{2}-\delta}\omega\|_\infty\lesssim \delta_0$. Picking $\delta_0<\frac{c}{100}$ will allow these terms to be absorbed into negative terms $-c\|r^\frac{-1+\delta}{2}v_t\|_{L^2_x(r<R)}^2 -c\|r^{-\frac{d-1}{2(d+1)}}v\|_{L^\frac{2(d+1)}{d-1}_x(r<R)}^\frac{2(d+1)}{d-1} -c\|r^{\frac{-3+\delta}{2}v}\|_{L^2_x(r<R)}^2$

For $r>R$, we estimate as following 
\begin{equation}\label{Term 2}
    \begin{aligned}
        \|v^\frac{4}{d-1}\omega v_t\|_{L^1_x(r>R)}&\lesssim \|e^{-r}v_t \cdot v^{\frac{4}{d-1}} \cdot e^r\omega\|_{L^1_x(r>R)}\\
        &\lesssim \|e^{-r}v_t\|_{L^2_x(r>R)}\|v\|_{\frac{2(d+1)}{d-1}}^\frac{4}{d-1}\|e^r\omega\|_{L^\frac{2(d+1)}{d-3}_x(r>R)}\\
        &\lesssim \delta_0\|e^{-r}v_t\|_{L^2_x(r>R)}^2 + \delta_0 \|v\|_\frac{2(d+1)}{d-1}^\frac{2(d+1)}{d-1} + C(\delta_0)\|e^r\omega\|_{L^\frac{2(d+1)}{d-3}_x(r>R)}^\frac{2(d+1)}{d-3}
    \end{aligned}
\end{equation}
Likewisde, picking $\delta_0<\frac{c}{100}$ will allow the negative terms to absorb the first two terms in (\ref{Term 2}). We have thus dealt with the first term $\|v^\frac{4}{d-1}\omega v_t\|_1$ in (\ref{Core Estimate}).

Next, we estimate the second positive term in (\ref{energy rate 2}): $\|v\omega^\frac{4}{d-1}v_t\|_1$. In fact, we can estimate it by
\begin{equation}
    \begin{aligned}
        \|v\omega^{\frac{4}{d-1}}v_t\|_{L^1_x(r>R)}&\lesssim \|v\|_\frac{2(d+1)}{d-1}\|e^{\frac{d-1}{4}r}\omega\|_{L^\frac{4(d+1)}{d-1}(r>R)}^\frac{4}{d-1}\|e^{-r}v_t\|_{L^2_x(r>R)}\\
        &\lesssim \delta_0\|v\|_\frac{2(d+1)}{d-1}^\frac{2(d+1)}{d-1} + \delta_0 \|e^{-r}v_t\|_{L^2_x(r>R)}^2 + C(\delta_0)\|e^{\frac{d-1}{4}r}\omega\|_{L^\frac{4(d+1)}{d-1}_x(r>R)}^\frac{4(d+1)}{d-1} \\
    \end{aligned}
\end{equation}

and 

\begin{equation}
    \begin{aligned}
        \|v\omega^\frac{4}{d-1}v_t\|_{L^1_x(r<R)} &\lesssim \|r^{\frac{-3+\delta}{2}}v\|_{L^2_x(r<R)} \|r^{\frac{d-1}{2}-\delta}\omega\|_{L^\infty_x(r<R)}^\frac{4}{d-1} \|r^{\frac{-1+\delta}{2}}v_t\|_{L^2_x(r<R)}\\
        &\lesssim \|r^{\frac{d-1}{2}-\delta}\omega\|_{L^\infty_x(r<R)}^\frac{4}{d-1} \|r^{\frac{-3+\delta}{2}}v\|_{L^2_x(r<R)}^2 + \|r^{\frac{d-1}{2}-\delta}\omega\|_{L^\infty_x(r<R)}^\frac{4}{d-1} \|r^{\frac{-1+\delta}{2}}v_t\|_{L^2_x(r<R)}^2
    \end{aligned}
\end{equation}
Again, picking $\delta_
0<\frac{c}{100}$ to conclude the estimate for the second term.

The third and fourth positive term in (\ref{energy rate 2}) can be estimated in the exact same way, due to $v_t$ and $\nabla v$ sharing similar estimates.

To conclude the proof of Proposition \ref{Core Estimate}, it only remains to estimate the fifth to eighth terms in (\ref{energy rate 2}). In fact, we have
\begin{equation}
    \begin{aligned}
        \|v^\frac{4}{d-1}\omega v\|_1 + \|v\omega^\frac{4}{d-1}v\|_1 \lesssim \delta_0\|v\|_\frac{2(d+1)}{d-1}^\frac{2(d+1)}{d-1} + C(\delta_0) \|\omega\|_\frac{2(d+1)}{d-1}^\frac{2(d+1)}{d-1}
    \end{aligned}
\end{equation}
For the seventh positive term in (\ref{energy rate 2}), we estimate the $r<R$ part by
\begin{equation}
    \begin{aligned}
        \|v^\frac{4}{d-1}\omega \frac{v}{r}\|_{L^1_x(r<R)} &\lesssim \|(r^{-3+\delta}v^2)^{(1+\alpha)} \cdot (r^{-(1-\delta)}v^\frac{2(d+1)}{d-1})^\beta\cdot (r^{\frac{d-1}{2}-\delta'}\omega) \|
    \end{aligned}
\end{equation}
and estimate it in the similar fashion as in (\ref{Term 1}).
For $r>R$, we apply Strichartz estimate which yields the following bound
\begin{equation}
    \begin{aligned}
        \|v^\frac{4}{d-1}\omega \frac{v}{r}\|_{L^1_x(r>R)}&\lesssim \|v\|_\frac{2(d+1)}{d-1}^2\|\omega\|_\frac{2(d+1)}{d-1}^\frac{4}{d-1}\\
        &\lesssim \delta_0 \|v\|_\frac{2(d+1)}{d-1}^\frac{2(d+1)}{d-1} + C(\delta_0)\|\omega\|_\frac{2(d+1)}{d-1}^\frac{2(d+1)}{d-1}
    \end{aligned}
\end{equation}

The eighth positive term in (\ref{energy rate 2}) can be estimated in the similar fashion.
Thus, the proof of Proposition \ref{Core Estimate} is concluded.

\end{proof}

\subsection{Case II: $d>5$} When $d>5$, we have $\frac{4}{d-1}<1$. And the term involving $v^\frac{4}{d-1}\omega$ cannot be estimated as before due to low power on v. In this case, we bound the error term $\mathcal{N}$ by
\begin{equation}
    |\mathcal{N}|=|u^\frac{d+3}{d-1}-\omega^\frac{d+3}{d-1}-v^\frac{d+3}{d-1}|\lesssim \inf \{|v\omega^\frac{4}{d-1}|,|v^\frac{4}{d-1}\omega|\}
\end{equation}
which means we can drop the $v^\frac{4}{d-1}\omega$ terms in (\ref{energy rate 2}). The remaining terms involving $v\omega^\frac{4}{d-1}$ can be estimated the same way as we did for $3\leq d \leq 5$.
This concludes the proof of theorem \ref{Main Thm}.

\section{Directions for future investigation}
In the final part of the paper, we make a few comments on the directions for further research of the global wellposedness and scattering problem of nonlinear wave equation on the hyperbolic space. 

1. \emph{Sharp results} We proved global wellposedness and scattering for equation (\ref{eq:nlw}) with initial data in $H^{2} \times H^{2-1}$, for any $s > \frac{1}{2}$. The conjecture derived from the scaling symmetry predicts positive results for initial data in $H^{\frac{1}{2}} \times H^{-\frac{1}{2}}$. We refer to \cite{dodson2018global} for a proof of the sharp result on Euclidean space. Despite the hyperbolic geometry greatly simplifying the proof for $s > \frac{1}{2}$, the author was not able to develop the necessary harmonic analysis tools on hyperbolic spaces similar to those used in \cite{dodson2018global} to treat low regularity critical problem on Euclidean spaces, e.g. an easy-to-use Littlewood Paley frequency decomposition.

2. \emph{Intercritical range} Intercritical wave equations between the conformal and energy critical state, namely with exponents $1 + \frac{4}{d-1} < p < 1 + \frac{4}{d-2}$, remains to be studied on non-euclidean spaces. We refer to \cite{dodson2023sharp} for results on Euclidean spaces. The author expects that the hyperbolic geometry should allow the proof of global wellposedness and scattering to be simplified, at least for non-sharp exponents $s > s_c$.

\bibliographystyle{amsplain}
\bibliography{ref.bib}

\end{document}